\documentclass[11 pt]{amsart}
\usepackage{eurosym}
\usepackage{amsfonts}
\usepackage{amssymb}
\usepackage{mathrsfs}

\newtheorem{thm}{Theorem}[subsection]
\newtheorem{cor}[thm]{Corollary}

\theoremstyle{definition}
\newtheorem{defn}[thm]{Definition}
\theoremstyle{remark}
\newtheorem{rem}[thm]{Remark}
\numberwithin{equation}{subsection}
\newtheorem{exmp}{Example}[section]

\begin{document}
\title[  ENRICHED INTERPOLATIVE KANNAN TYPE OPERATORS  ]{
 FIXED POINT RESULTS OF ENRICHED INTERPOLATIVE KANNAN TYPE OPERATORS WITH APPLICATIONS }
\author{Mujahid Abbas}
\address{Department of Mathematics, Government College University Katchery
Road, Lahore 54000, Pakistan and Department of Mathematics and Applied
Mathematics, University of Pretoria Hatfield 002, Pretoria, South Africa}
\email{mujahid.abbas@up.ac.za}
\author{Rizwan Anjum}
\address{Department of Mathematics, Riphah Institute of Computing and Applied Sciences, Riphah International University, Lahore, Pakistan.}
\email{rizwan.anjum@riphah.edu.pk}
\author{Shakeela Riasat}
\address{Abdus Salam School of Mathematical Sciences, Government College
University, Lahore 54600, Pakistan.}
\email{shakeelariasat@sms.edu.pk}
\thanks{}
\thanks{}
\date{}
\dedicatory{}

\begin{abstract}
The purpose of this paper is to introduce the class of enriched
interpolative Kannan type operators on Banach space that contains the
classes of enriched Kannan operators, interpolative Kannan type contraction
operators and some other classes of nonlinear operators. Some examples are
presented to support the concepts introduced herein. A convergence theorem
for the Krasnoselskii iteration method to approximate fixed point of the
enriched interpolative Kannan type operators is proved. We study
well-posedness, Ulam-Hyers stability and periodic point property of
operators introduced herein. As an application of the main result,
variational inequality problem is solved.\newline

\vspace{0,1cm} \noindent 2020 \textit{Mathematics Subject Classification.}
47H10. 47H09. 47J25. 49J40

\vspace{0,1cm} \noindent \textit{Keywords: Fixed point, enriched Kannan
operators, interpolative Kannan type contraction, Krasnoselskii iteration,
well-posedness, periodic point, Ulam-Hyers stability, variational inequality problem }
\end{abstract}

\maketitle

\commby{}

\section{Introduction and Preliminaries}

Let $(X,d)$ be a metric space and $T:X\rightarrow X$. We denote the set $%
\{x\in X:Tx=x\}$ of fixed points of $T$ by ${Fix}(T)$. Solving a fixed point
problem of an operator $T,$ denote by $FPP(T)$ is to show that the set ${Fix}%
(T)$ is nonempty. \newline
Define $T^{0}=I$ (the identity map on $X$) and $T^{n}=T^{n-1}\circ T$,
called the $n^{th}$-iterate of $T$ for $n\geq 1$. The most simplest
iteration procedure to approximate the solution of a fixed point equation $%
Tx=x$ is the method of successive approximations (or Picard iteration) given
by

\begin{equation}
x_{n}=T^{n}x_{0},\ \ n=1,2,\dotsc ,  \label{newpic}
\end{equation}%
where $x_{0}$ is an initial guess in domain of an operator $T.$ \newline
A sequence $\{x_{n}\}_{n=0}^{\infty }$ in $X$ given by (\ref{newpic}) is
called a $T$-orbital sequence around $x_{0}.$ The collection of all such
sequences is denoted by $O(T,x_{0})$. If there exists $x^{\ast }$ in $X$
such that ${Fix}(T)=\{x^{\ast }\}$ and the Picard iteration associated with $%
T$ converges to $x^{\ast }$ for any initial guess $x_{0}$ in $X$, then $T$
is called a Picard operator (see \cite{Anjum,Agarwal2,Rus}). \newline
If there exists $k\in \lbrack 0,1)$ such that for any $x,y\in X,$ we have
\begin{equation}
d(Tx,Ty)\leq k\,d(x,y).  \label{ban}
\end{equation}%
Then $T$ is called a Banach contraction operator which, if defined on a
complete metric space, is a classical example of a continuous Picard
operator. Thus, it was natural to ask the question whether in the framework
of an complete metric space, a discontinuous operator satisfying somewhat
similar contractive conditions is a Picard operator. This was answered in an
affirmative by Kannan \cite{Kannan} in $1968$. \newline
An operator $T$ on $X$ is called Kannan contraction operator if there exists
$a\in \lbrack 0,0.5)$ such that for any $x,y\in X,$ we have
\begin{equation}
d(Tx,Ty)\leq a\{d(x,Tx)+d(y,Ty)\}.  \label{Kannancont}
\end{equation}%
\newline
Kannan contraction operator defined on a complete metric space is an example
of a discontinuous Picard operator (\cite{Kannan1}, \cite{Kannan}).

Subrahmanyam \cite{Subrahmanyam} proved that a metric space $X$ is complete
if and only if every Kannan contraction operator on $X$ has a fixed point.
Moreover, Connell \cite{Connell} gave an example of an incomplete metric
space $X$ on which every Banach contraction operator has a fixed point. This
shows that the Banach contraction operators do not characterize the
completeness of their domain. As long as contractive conditions are
concerned, the classes of Banach and Kannan contraction operators are
incomparable (\cite{Kikkawa}), but the class of Kannan contraction operators attracted the attention of several mathematicians because of its connection
with a characterization of its metric completeness.\newline
Kannan's theorem has been generalized in different ways by many authors to
extend the limits of metric fixed point theory in different directions.
\newline
Karapinar introduced a new class of Kannan type operators called
interpolative Kannan type operators and proved a fixed point result for such
operators in the setup of complete metric spaces (\cite{Karapinar}).\newline
An operator $T:X\rightarrow X$ is called an interpolative Kannan type if
there exists $a\in \lbrack 0,1)$ such that for all $x,y\in X\setminus Fix(T),
$ we \ have
\begin{equation}
d(Tx,Ty)\leq a[d(x,Tx)]^{\alpha }[d(y,Ty)]^{1-\alpha },  \label{IKC}
\end{equation}%
where $\alpha \in (0,1).$\newline
The main result in \cite{Karapinar} is stated as follows.

\begin{thm}
\cite{Karapinar}\label{098765} Let $(X,d)$ be a complete metric space and $T$
an interpolative Kannan type operator. Then $T$ is a Picard operator.
\end{thm}

For more results in this direction, we refer to
\cite{Agarwal,Aydi,Aydi2,Chifu,Debnath,Gaba,Gaba2,Karapinar2,Karapinar3,Karapinar4,K1,K2,K3,K4,K5,Noorwali})
and references mentioned therein.\newline
Existence, uniqueness, stability, approximation and characterization of
fixed points of certain operators are some of the main concerns of a metric
fixed point theory. Contractive conditions on operators and distance
structure of operator's domain play a vital role to ensure the convergence
of iterative methods.\newline
If for any $x,y\in X,$ we have%
\begin{equation*}
d(Tx,Ty)\leq d(x,y).
\end{equation*}%
\newline
Then $T$ is called a nonexpansive operator. An operator $T$ on $X$ is called
asymptotically regular on $X$ if for all $x$ in $X,$%
\begin{equation*}
d(T^{n+1}x,T^{n}x)\rightarrow 0\text{ as }n\rightarrow \infty .
\end{equation*}%
The concept of asymptotic regularity plays an important role in
approximating the fixed points of operators. Picard iteration method fails to
converge to a fixed point of certain contractive mappings such as
nonexpansive mappings on metric spaces. This led to the study of a variety of fixed point iteration procedures in the setup of Banach spaces.\newline
In this paper, we shall approximate the fixed point of some nonlinear
mappings through Krasnoselskii iteration method. \newline
Let $\lambda \in \lbrack 0,1].$ A sequence $\{x_{n}\}_{n=0}^{\infty }$ given
by
\begin{equation}
x_{n+1}=(1-\lambda )x_{n}+\lambda Tx_{n},\ \ n=0,1,2,\dotsc   \label{newKar}
\end{equation}%
is called the Krasnoselskii iteration. \newline
Note that, Krasnoselskii iteration $\{x_{n}\}_{n=0}^{\infty }$ sequence
given by (\ref{newKar}) is exactly the Picard iteration corresponding to an
averaged operator
\begin{equation}
T_{\lambda }=(1-\lambda )I+\lambda T.  \label{E7}
\end{equation}%
Moreover, for $\lambda =1$ the Krasnoselskii iteration method reduces to
Picard iteration method. Also, $Fix(T)=Fix(T_{\lambda })$, for all $\lambda
\in (0,1].$\newline
On the other hand, Browder and Petryshyn \cite{article.6} introduced the
concept of asymptotic regularity in connection with the study of fixed
points of nonexpansive mappings. As a matter of fact, the same property was
used in 1955 by Krasnoselskii \cite{Krasnoselskii} to prove that if $K$ is a
compact convex subset of a uniformly convex Banach space and $T:K\rightarrow
K$ is a nonexpansive mapping, then for any $x_{0}\in K,$ the sequence
\begin{equation}
x_{n+1}=\frac{1}{2}(x_{n}+Tx_{n}),\ n\geq 0,  \label{lu}
\end{equation}%
converges to fixed point of $T.$\newline
In proving the above result, Krasnoselskii used the fact that if $T$ is
nonexpansive which, in general, is not asymptotically regular, then the
averaged mapping $T_{\frac{1}{2}}$ is asymptotically regular.\newline
Therefore, an averaged operator $T_{\lambda }$ enriches the class of
nonexpansive mappings with respect to the asymptotic regularity. This
observation suggested that one could enrich the classes of contractive
mappings studied in the framework of metric spaces by imposing certain
contractive condition on $T_{\lambda }$ instead of $T$ itself.\newline
Employing this approach, the classes such as enriched contractions and
enriched $\phi $-contractions \cite{B4}, enriched Kannan contractions \cite%
{B6}, enriched Chatterjea mappings \cite{B7}, enriched nonexpansive mappings
in Hilbert spaces \cite{B2}, enriched multivalued contractions \cite{Abbas2}
and enriched \'{C}iri\'{c}-Reich-Rus contraction \cite{B9}, enriched cyclic
 contraction \cite{Anjum3new}, enriched modified Kannan pair \cite{AnjumNew},  enriched quasi
contraction \cite{Abbas22} were introduced and studied.\newline
Abbas et al. \cite{Abbas2} proved fixed point results by imposing the
condition that $T_{\lambda }$-orbital subset is a complete subset of a
normed space (see, Theorem 3 of \cite{Abbas2}). Similarly, G\'{o}rnicki and Bisht\cite{GronickiBisht} considered the enriched \'{C}iri\'{c}%
-Reich-Rus contraction operators and proved a fixed point theorem by
imposing the condition that $T_{\lambda }$ is asymptotically regular mapping
(see, Theorem 3.1 of \cite{GronickiBisht}).\newline

Consistent with \cite{B6}, let $(X,\left\Vert \cdot \right\Vert )$ be a
normed space. A operator $T:X\rightarrow X$ is called an enriched Kannan
contraction or $(b,a)$-enriched Kannan contraction if there exist $b\in
\lbrack 0,\infty )$ and $a\in \lbrack 0,0.5)$ such that for all $x,y\in X,$
the following holds:
\begin{equation}
\left\Vert b(x-y)+Tx-Ty\right\Vert \leq a(\left\Vert x-Tx\right\Vert
+\left\Vert y-Ty\right\Vert ).  \label{enrichedKannan}
\end{equation}%
As shown in \cite{B6}, several well-known contractive conditions in the
existing literature imply the $(b,a)$-enriched Kannan contraction condition.
It was proved in \cite{B6} that any enriched Kannan contraction operator
defined on a Banach space has a unique fixed point which can be approximated
by means of the Krasnoselskii iterative scheme.\newline

Motivated by the work of Berinde and P\u{a}curar \cite{B6}, Absas et al.
\cite{Abbas2} and G\'{o}rnicki and Bisht \cite{GronickiBisht}, we
propose a new class of enriched interpolative Kannan type operators. The
purpose of this paper is to prove the existence of fixed point of such
operators. Moreover, we study the well-posedness, Ulam-Hyers stability and
periodic point property of the operators introduced herein. Finally, an
application of our result to solve variational inequality problems is also
given.

\section{Fixed point approximation of enriched interpolative Kannan type
operators}

In the sequel, the notations $\mathbb{N}$ and $\mathbb{R}$ will denote the
set of all natural numbers and the set of all real numbers, respectively.%
\newline
In this section, we present a new class of enriched interpolative Kannan
type operators, which is first of its kind in the existing literature on
metric fixed-point theory. Existence and convergence results of such class
of operators are also obtained. \newline
First, we introduce the following concept

\begin{defn}
Let $(X,\left\Vert \cdot \right\Vert )$ be a normed space. A mapping $%
T:X\rightarrow X$ is called enriched interpolative Kannan type operator if
there exist $b\in \lbrack 0,\infty ),\ a\in \lbrack 0,1)$ and $\alpha \in
(0,1)$ such that for all $x,y\in X,$ we have
\begin{equation}
\left\Vert b(x-y)+(Tx-Ty)\right\Vert \leq a\big(\left\Vert x-Tx\right\Vert %
\big)^{\alpha }\big(\left\Vert y-Ty\right\Vert \big)^{1-\alpha }.  \label{sha}
\end{equation}%
To highlight an involvement of parameters $a,b$ and $\alpha $ in (\ref{sha}%
), we call $T$ a $(a,b,\alpha )$-enriched interpolative Kannan type operator.
\end{defn}

\begin{exmp}
Any interpolative Kannan type operator $T$ satisfying (\ref{IKC}) is a $%
(a,0,\alpha )$-enriched interpolative Kannan contraction operator, that is, $%
T$ satisfies (\ref{sha}) with $b=0.$
\end{exmp}

We now give an example of an enriched interpolative Kannan type operator
which is not a interpolative Kannan type contraction operator.

\begin{exmp}
\label{exam} Let $(Y,\mu )$ be a finite measure space. The classical
Lebesgue space $X=L^{2}(Y,\mu )$ is defined as the collection of all Borel
measurable functions $f:Y\rightarrow \mathbb{R}$ such that
\mbox{$\int_{Y}|f(y)|^{2}d\mu(y)
<\infty.$} We know that the space $X$ equipped with the norm $\Vert f\Vert
_{X}=\bigg(\int_{Y}|f|^{2}d\mu \bigg)^{\frac{1}{2}}$ is a Banach space.
\newline
Define the operator $T:L^{2}(Y,\mu )\rightarrow L^{2}(Y,\mu )$ by
\begin{equation*}
Tf=g-3f,
\end{equation*}%
where $g(y)=1,\ \forall \ y\in Y.$ Clearly, $g\in L^{2}(Y,\mu )$ as $\mu
(Y)<\infty .$\newline
Note that $T$ is an $(0.5,3,0.5)$-enriched interpolative Kannan type
operator but not an interpolative Kannan type contraction operator. \newline
Indeed, if $T$ would be a interpolative Kannan type operator then, by (\ref%
{IKC}), there would exist $a\in \lbrack 0,1)$ and $\alpha \in (0,1)$ such
that
\begin{equation*}
\left\Vert -3f+3h\right\Vert _{X}\leq a\big(\left\Vert 4f-g\right\Vert _{X}%
\big)^{\alpha }\big(\left\Vert 4h-g\right\Vert _{X}\big)^{1-\alpha }\ \
\forall \ f,h\in X,
\end{equation*}%
which on taking $f(y)=0$ and $h(y)=1,$ for all$\ y\in Y$ gives $3^{\alpha
}\leq a<1,$ a contradiction$.$\newline
\end{exmp}

We now present the following result.

\begin{thm}
\label{main} Let $(X,\left\Vert \cdot \right\Vert )$ be a normed space, $%
T:X\rightarrow X$ a $(a,b,\alpha )$-enriched interpolative Kannan type
mapping.\newline
Then,

\begin{enumerate}
\item $Fix(T)=\{ x^{\ast}\};$

\item There exist a $T_{\lambda }$-orbital sequence $\{x_{n}\}_{n=0}^{\infty
}$ around $x_{0}$, given by
\begin{equation}
x_{n+1}=(1-\lambda )x_{n}+\lambda Tx_{n};\ \ n\geq 0,  \label{2}
\end{equation}%
converges to $x^{\ast }$ provided that, for $x_{0}\in X$, $T_{\lambda }$%
-orbital subset $O(T_{\lambda },x_{0})$ is a complete subset of $X$, where $%
\lambda =\frac{1}{b+1}.$
\end{enumerate}
\end{thm}

\begin{proof}
We divide the proof into the following two cases.\newline
Case $1$. If $b>0$. Then $\lambda =\frac{1}{b+1}\in (0,1)$ and the enriched
interpolative Kannan type operator (\ref{sha}) satisfies the following
contraction condition:
\begin{equation*}
\left\Vert \bigg(\frac{1}{\lambda }-1\bigg)(x-y)+Tx-Ty\right\Vert \leq
a\left\Vert x-Tx\right\Vert ^{\alpha }\left\Vert y-Ty\right\Vert ^{1-\alpha }
\end{equation*}%
and hence
\begin{equation*}
\left\Vert (1-\lambda )(x-y)+\lambda (Tx-Ty)\right\Vert \leq a\lambda
\left\Vert x-Tx\right\Vert ^{\alpha }\left\Vert y-Ty\right\Vert ^{1-\alpha }
\end{equation*}%
which can be written in an equivalent form as follows:
\begin{equation}
\left\Vert T_{\lambda }x-T_{\lambda }y\right\Vert \leq a\left\Vert
x-T_{\lambda }x\right\Vert ^{\alpha }\left\Vert y-T_{\lambda }y\right\Vert
^{1-\alpha },\ \ \forall \ x,y\in X.  \label{3}
\end{equation}%
In view of (\ref{E7}), the Krasnoselskii iterative sequence defined by (\ref%
{2}) is exactly the Picard's iteration associated with $T_{\lambda }$, that
is,
\begin{equation}
x_{n+1}=T_{\lambda }x_{n},\ \ n\geq 0.  \label{7}
\end{equation}%
Take $x:=x_{n}$ and $y:=x_{n-1}$ in (\ref{3}) to get
\begin{align*}
\left\Vert x_{n+1}-x_{n}\right\Vert & =\left\Vert T_{\lambda
}x_{n}-T_{\lambda }x_{n-1}\right\Vert \\
& \leq a\left\Vert x_{n}-T_{\lambda }x_{n}\right\Vert ^{\alpha }\left\Vert
x_{n-1}-T_{\lambda }x_{n-1}\right\Vert ^{1-\alpha } \\
& \leq a\left\Vert x_{n}-x_{n+1}\right\Vert ^{\alpha }\left\Vert
x_{n-1}-x_{n}\right\Vert ^{1-\alpha }
\end{align*}%
which implies that
\begin{equation}
\left\Vert x_{n+1}-x_{n}\right\Vert ^{1-\alpha }\leq a\left\Vert
x_{n-1}-x_{n}\right\Vert ^{1-\alpha }.  \label{89}
\end{equation}%
As $\alpha \in (0,1),$
\begin{equation}
\left\Vert x_{n+1}-x_{n}\right\Vert \leq a\left\Vert
x_{n-1}-x_{n}\right\Vert .  \label{ri}
\end{equation}%
Inductively, we obtain that
\begin{equation}
\left\Vert x_{n+1}-x_{n}\right\Vert \leq a^{n}\left\Vert
x_{0}-x_{1}\right\Vert .  \label{riz}
\end{equation}%
By (\ref{riz}) and triangular inequality, we have
\begin{equation}
\left\Vert x_{n}-x_{n+r}\right\Vert \leq \frac{a^{n}}{1-a}\left\Vert
x_{0}-x_{1}\right\Vert ,\ \ r\in \mathbb{N},\ n\geq 1,  \label{wer}
\end{equation}%
which, in view of $0<a<1$ gives that $\{x_{n}\}_{n=0}^{\infty }$ is a Cauchy
sequence in the complete subset $O(T_{\lambda },x_{0})$ of $X.$ \newline
Next, we assume that there exists an element $x^{\ast }$ in $O(T_{\lambda
},x_{0})$ such that $\lim_{n\rightarrow \infty }x_{n}=x^{\ast }.$ Note that
\begin{align*}
\left\Vert x^{\ast }-T_{\lambda }x^{\ast }\right\Vert & \leq \left\Vert
x^{\ast }-x_{n+1}\right\Vert +\left\Vert x_{n+1}-T_{\lambda }x^{\ast
}\right\Vert \\
& \leq \left\Vert x^{\ast }-x_{n+1}\right\Vert +\left\Vert T_{\lambda
}x_{n}-T_{\lambda }x^{\ast }\right\Vert \\
& \leq \left\Vert x^{\ast }-x_{n+1}\right\Vert +a\left\Vert x_{n}-T_{\lambda
}x_{n}\right\Vert ^{\alpha }\left\Vert x^{\ast }-T_{\lambda }x^{\ast
}\right\Vert ^{1-\alpha } \\
& \leq \left\Vert x^{\ast }-x_{n+1}\right\Vert +a\left\Vert
x_{n}-x_{n+1}\right\Vert ^{\alpha }\left\Vert x^{\ast }-T_{\lambda }x^{\ast
}\right\Vert ^{1-\alpha }.
\end{align*}

On taking the limit as $n\rightarrow \infty $ on both sides of the above
inequality$,$ we get that $x^{\ast }=T_{\lambda }x^{\ast }.$\newline

Assume that $x^{\ast }$ and $y^{\ast }$ are two fixed points of $T.$ Then
from (\ref{3}), we have
\begin{align*}
\left\Vert x^{\ast }-y^{\ast }\right\Vert & =\left\Vert T_{\lambda }x^{\ast
}-T_{\lambda }y^{\ast }\right\Vert \leq a\left\Vert x^{\ast }-T_{\lambda
}x^{\ast }\right\Vert ^{\alpha }\left\Vert y^{\ast }-T_{\lambda }y^{\ast
}\right\Vert ^{1-\alpha } \\
& \leq a\left\Vert x^{\ast }-x^{\ast }\right\Vert ^{\alpha }\left\Vert
y^{\ast }-y^{\ast }\right\Vert ^{1-\alpha },
\end{align*}%
which, gives $x^{\ast }=y^{\ast }.$ \newline
Case 2. $b=0.$ In this case, the enriched interpolative Kannan type operator
(\ref{sha}) becomes
\begin{equation}
\left\Vert Tx-Ty\right\Vert \leq a\left\Vert x-Tx\right\Vert ^{\alpha
}\left\Vert y-Ty\right\Vert ^{1-\alpha }\ \ \ \forall \ x,y\in X,  \label{b}
\end{equation}%
where $a\in (0,1).$ That is, $T$ is an interpolative Kannan type contraction
operator and hence by Theorem 2.2 of \cite{Karapinar}, $T$ has a unique
fixed point.
\end{proof}

The following example illustrate the above theorem.

\begin{exmp}
Let $X=\mathbb{R}\setminus \big\{\frac{1}{5},\frac{4}{5}\big\}$ be endowed
with the usual norm and $T:X\rightarrow X$ be defined by $Tx=1-x,$ $\forall
\ x\in X$. Clearly $T$ is a $(1,1)$-enriched interpolative Kannan type operator. Now $%
\lambda =\frac{1}{1+b}$ gives that $\lambda =\frac{1}{2}$. Let $x_{0}=1/2$
be the fixed in $X$. Then, we have
\begin{equation*}
T_{\frac{1}{2}}(x_{0})=\frac{1}{2}x_{0}+\frac{1}{2}Tx_{0}=\frac{1}{2}\bigg(%
\frac{1}{2}\bigg)+\frac{1}{2}\bigg(\frac{1}{2}\bigg)=\frac{1}{2}.
\end{equation*}%
Pick $x_{1}=T_{\frac{1}{2}}x_{0}=1/2.$ Continuing this way, we obtained $%
x_{n+1}=T_{\frac{1}{2}}x_{n},$ where $x_{n+1}=(\frac{1}{2},\frac{1}{2},\frac{%
1}{2},\dotsc )$. Note that $\{x_{n}\}_{n=0}^{\infty }$ is a Cauchy sequence
which converges to $1/2$ and $1/2$ is the fixed point of $T.$
\end{exmp}

We now present the following fixed point theorem for $(a,b,\alpha )$%
-enriched interpolative Kannan type operator on a Banach space.

\begin{cor}
\label{Shak} Let $(X,\left\Vert \cdot \right\Vert )$ be a Banach space and $%
T:X\rightarrow X$ a $(a,b,\alpha )$-enriched interpolative Kannan type
operator. Then $T$ has a unique fixed point.\end{cor}

\begin{proof}
Following arguments similar to those in proof Theorem \ref{main}, the result
follows.
\end{proof}

If we take $b=0$ in the Corollary \ref{Shak}, we obtain Theorem $2.2$ of
\cite{Karapinar} in the setting of Banach spaces.

\begin{cor}
\cite{Karapinar} Let $T$ be an interpolative Kannan type operator on a
Banach space $(X,\left\Vert \cdot \right\Vert ).$ Then $T$ is a Picard operator.
\end{cor}

By Corollary \ref{Shak}, we obtain the following result.

\begin{cor}
\cite{B6} Let $(X,\left\Vert \cdot \right\Vert )$ be a Banach space and $%
T:X\rightarrow X$ an $(b,a)$-enriched Kannan contraction, that is, for all $%
x,y\in X,$ it satisfies the following inequality;
\begin{equation}
\left\Vert b(x-y)+Tx-Ty\right\Vert \leq a\big\{\left\Vert x-Tx\right\Vert
+\left\Vert y-Ty\right\Vert \big\}  \label{EnrichedKannan}
\end{equation}%
with $b\in \lbrack 0,\infty )$ and $a\in \lbrack 0,0.5).$ Then $T$ has a
unique fixed point.
\end{cor}

\begin{proof}
Take $\lambda =\frac{1}{b+1}.$ Obviously, $0<\lambda <1$ and the $(b,a)$%
-enriched Kannan contraction condition (\ref{EnrichedKannan}) becomes
\begin{eqnarray*}
&&\left\Vert \bigg(\frac{1}{\lambda }-1\bigg)(x-y)+Tx-Ty\right\Vert \\
&\leq &a\big\{\left\Vert x-Tx\right\Vert +\left\Vert y-Ty\right\Vert \big\}%
,\ \forall \ x,y\in X,
\end{eqnarray*}%
which can be written in an equivalent form as follows;
\begin{equation}
\left\Vert T_{\lambda }x-T_{\lambda }y\right\Vert \leq a\big\{\left\Vert
x-T_{\lambda }x\right\Vert +\left\Vert y-T_{\lambda }y\right\Vert \big\},\
\forall \ x,y\in X.  \label{used11}
\end{equation}%
By (\ref{used11}), $T_{\lambda }$ is a Kannan contraction. It follows from
\cite{Karapinar} that $T_{\lambda }$ satisfies condition (\ref{used11}) and
condition (\ref{3}). Since, for $\lambda =\frac{1}{b+1},$ the inequality (%
\ref{3}) is same as the condition (\ref{sha}). This suggests that $T$ is an
enriched interpolative Kannan type operator and then the Corollary \ref{Shak}
leads to the conclusion.
\end{proof}

\section{Well-Posedness, Perodic Point and Ulam-Hyers Stability Results}

We now present the well-Posedness, perodic point and Ulam-Hyers stability
results for $(a,b,\alpha )$-enriched interpolative Kannan type operators.

\subsection{Well-Posedness}

\ \newline
Let us start with the following definition.

\begin{defn}
\cite{article.11} Let $(X,d)$ be a metric space and $T:X\rightarrow X$. The
fixed point problem $FPP(T)$ is said to be well-posed if $T$ has unique
fixed point $x^{\ast }($say) and for any sequence $\{x_{n}\}$ in $X$
satisfying $\lim_{n\rightarrow \infty }d(Tx_{n},x_{n})=0,$ we have $%
\lim_{n\rightarrow \infty }x_{n}=x^{\ast }.$
\end{defn}

Since $Fix(T)=Fix(T_{\lambda })$, we conclude that the fixed point problem
of $T$ is well-posed if and only if the fixed point problem of $T_{\lambda }$
is well-posed.\newline
Well-posedness of certain fixed point problems has been studied by several
mathematicians, see for example, \cite{article.5}, \cite{article.11} and
references mentioned therein.\newline
We now study the well-posedness of a fixed point problem of mappings in
Theorem \ref{main} and Corollary \ref{Shak}.

\begin{thm}
\label{math} Let $(X,\left\Vert \cdot \right\Vert )$ be a Banach space.
Suppose that $T$ is an operator on $X$  as in the Theorem \ref{main}.
Then, $FPP(T)$ is well-posed.
\end{thm}

\begin{proof}
It follows from Theorem \ref{main} that $x^{\ast }$ is the unique fixed
point of $T.$ Suppose that $\lim_{n\rightarrow \infty }\left\Vert T_{\lambda
}x_{n}-x_{n}\right\Vert =0.$ Using (\ref{3}) we have,
\begin{align*}
\left\Vert x_{n}-x^{\ast }\right\Vert & \leq \left\Vert x_{n}-T_{\lambda
}x_{n}\right\Vert +\left\Vert T_{\lambda }x_{n}-x^{\ast }\right\Vert \\
& =\left\Vert x_{n}-T_{\lambda }x_{n}\right\Vert +\left\Vert T_{\lambda
}x_{n}-T_{\lambda }x^{\ast }\right\Vert \\
& \leq \left\Vert x_{n}-T_{\lambda }x_{n}\right\Vert +a\left\Vert
x_{n}-T_{\lambda }x_{n}\right\Vert ^{\alpha }\left\Vert x^{\ast }-T_{\lambda
}x^{\ast }\right\Vert ^{1-\alpha }
\end{align*}%
that is,
\begin{equation}
\left\Vert x_{n}-x^{\ast }\right\Vert \leq \left\Vert x_{n}-T_{\lambda
}x_{n}\right\Vert .  \label{nn}
\end{equation}%
It follows from (\ref{nn}) that $\lim_{n\rightarrow \infty }x_{n}=x^{\ast }$
provided that $\lim_{n\rightarrow \infty }\left\Vert T_{\lambda
}x_{n}-x_{n}\right\Vert =0.$ This complete the proof.
\end{proof}

\begin{cor}
Let $(X,\left\Vert \cdot \right\Vert )$ be Banach space. Suppose that $T$ is
an operator on $X$ as in the Corollary \ref{Shak}. Then the fixed point
problem is well-posed.
\end{cor}

\begin{proof}
Following arguments similar to those in the proof Theorem \ref{math}, the
result follows.
\end{proof}

\subsection{Perodic Point Result}

\ Clearly, a fixed point $x^{\ast }$ of $T$ is also a fixed point of $T^{n}$
for every $n\in \mathbb{N}$. However, the converse is false. For example, if
we take, $X=[0,1]$ and define an operator $T$ \ on $X$ by $Tx=1-x.$ Then $T$
has a unique fixed point $1/2,$ and for each even integer $n,$ $n^{th}$-iterate
of $T$ is an identity map and hence every point of $[0,1]$ is a fixed point
of $T^{n}$. Also, if $X=[0,\pi ],$ $Tx=\cos x,$ then every iterate of $T$
has the same fixed point as $T.$\newline
If a map T satisfies $Fix(T)=Fix(T^{n})$ for each $n\in \mathbb{N}$, then it
is said to have a periodic point property $P$ \cite{Jeong}.\newline
Since $Fix(T)=Fix(T_{\lambda })$, we conclude that the mapping $T$ has
property $P$ if and only the mapping $T_{\lambda }$ has property $P.$

\begin{thm}
\label{math2} Let $(X,\left\Vert .\right\Vert )$ be a Banach space. Suppose
that $T$ is an operator on $X$ as in the Theorem \ref{main}. Then $T$ has
property $P.$
\end{thm}

\begin{proof}
From Theorem \ref{main}, $T$ has a fixed point. Let $y^{\ast }\in
Fix(T^{n}). $ Now from (\ref{3}), we have
\begin{align*}
\left\Vert y^{\ast }-T_{\lambda }y^{\ast }\right\Vert & =\left\Vert
T_{\lambda }^{n}y^{\ast }-T_{\lambda }\big(T_{\lambda }^{n}y^{\ast }\big)%
\right\Vert \\
& =\left\Vert T_{\lambda }\big(T_{\lambda }^{n-1}y^{\ast }\big)-T_{\lambda }%
\big(T_{\lambda }^{n}y^{\ast }\big)\right\Vert \\
& \leq a\left\Vert T_{\lambda }^{n-1}y^{\ast }-T_{\lambda }^{n}y^{\ast
}\right\Vert ^{\alpha }\left\Vert T_{\lambda }^{n}y^{\ast }-T_{\lambda
}^{n+1}y^{\ast }\right\Vert ^{1-\alpha },
\end{align*}%
that is,
\begin{equation}
\left\Vert T_{\lambda }^{n}y^{\ast }-T_{\lambda }^{n+1}y^{\ast }\right\Vert
^{\alpha }\leq a\left\Vert T_{\lambda }^{n-1}y^{\ast }-T_{\lambda
}^{n}y^{\ast }\right\Vert ^{\alpha }.  \label{aza}
\end{equation}%
Since $\alpha \in (0,1),$ then (\ref{aza}) becomes
\begin{align*}
\left\Vert y^{\ast }-T_{\lambda }y^{\ast }\right\Vert & =\left\Vert
T_{\lambda }^{n}y^{\ast }-T_{\lambda }^{n+1}y^{\ast }\right\Vert \leq
a\left\Vert T_{\lambda }^{n-1}y^{\ast }-T_{\lambda }^{n}y^{\ast }\right\Vert
\leq \\
& \leq a^{2}\left\Vert T_{\lambda }^{n-2}y^{\ast }-T_{\lambda }^{n-1}y^{\ast
}\right\Vert \leq \dotsc \leq a^{n}\left\Vert y^{\ast }-T_{\lambda }y^{\ast
}\right\Vert .
\end{align*}%
Now, $0\leq a<1$ implies that $\left\Vert y^{\ast }-T_{\lambda }y^{\ast
}\right\Vert =0$ and hence $y^{\ast }=Ty^{\ast }.$
\end{proof}

\subsection{Ulam-Hyers Stability}

Let $(X,d)$ be a metric space, $T:X\rightarrow X$ and $\epsilon >0$ . A
point $w^{\ast }\in X$ called an $\epsilon $-solution of the fixed point
problem $FPP(T)$ if $w^{\ast }$ satisfies the following inequality%
\begin{equation*}
d(w^{\ast },Tw^{\ast })\leq \epsilon .
\end{equation*}

Let us recall the notion of Ulam-Hyers stability.

\begin{defn}
\cite{Sintunavarat} Let $(X,d)$ be a metric space, $T:X\rightarrow X$ and $%
\epsilon >0$. The fixed point problem $FPP(T)$ is called generalized
Ulam-Hyers stable if and only if there exists an increasing and continuos
function $\phi :[0,\infty )\rightarrow \lbrack 0,\infty )$ with $\phi (0)=0$
such that for each $\epsilon $-solution $w^{\ast }\in X$ of the fixed point
equation $Tx=x$ $,$ there exists a solution $x^{\ast }$ of $Tx=x$ in $X$
such that
\begin{equation*}
d(x^{\ast },w^{\ast })\leq \phi (\epsilon ).
\end{equation*}
\end{defn}

\begin{rem}
If the function $\phi $ in the above definition is given by $\phi (t)=mt$
for all $t\geq 0$, where $m>0,$ then the fixed point equation $Tx=x$ is said
to be Ulam-Hyers stable.
\end{rem}

\begin{thm}
\label{math3} Let $(X,\left\Vert \cdot \right\Vert )$ be a Banach space.
Suppose that $T$ is an operator on $X$ as in the Theorem \ref{main}. Then
the fixed point problem is Ulam-Hyers stable.
\end{thm}

\begin{proof}
Since $Fix(T)=Fix(T_{\lambda }),$ it follows that the fixed point problem $%
Tx=x$ is equivalent to the fixed point problem
\begin{equation}
x=T_{\lambda }x.  \label{stabTlambda}
\end{equation}%
Let $w^{\ast }$ be $\epsilon $-solution of the fixed point equation (\ref%
{stabTlambda}), that is,
\begin{equation}
d(w^{\ast },T_{\lambda }w^{\ast })\leq \epsilon .  \label{newwwww}
\end{equation}%
Using (\ref{3}) and (\ref{newwwww}), we get
\begin{align*}
\left\Vert x^{\ast }-w^{\ast }\right\Vert & =\left\Vert T_{\lambda }x^{\ast
}-w^{\ast }\right\Vert \leq \left\Vert T_{\lambda }x^{\ast }-T_{\lambda
}w^{\ast }\right\Vert +\left\Vert T_{\lambda }w^{\ast }-w^{\ast }\right\Vert
\\
& \leq a\left\Vert x^{\ast }-T_{\lambda }x^{\ast }\right\Vert ^{\alpha
}\left\Vert w^{\ast }-T_{\lambda }w^{\ast }\right\Vert ^{1-\alpha }+\epsilon
\\
& =\epsilon .
\end{align*}
\end{proof}

\section{Application to Variational inequality problem}

Variational inequality theory provides some important tools to handle the
problems arising in economic, engineering, mechanics, mathematical
programming, transportation and others. Many numerical methods have been
constructed for solving variational inequalities and optimization problems.
The aim of this section is to present generic convergence theorems for
Krasnoselskii type algorithms that solve variational inequality problems.
\newline

Let $H$ be a real Hilbert space with inner product $\big<\cdot ,\cdot \big>$%
, $C\subset H$ a closed and convex set and $S:H\rightarrow H$. \newline
The variational inequality problem with respect to $S$ and $C,$ denoted by $%
VIP(S,C),$ is to find $x^{\ast }\in C$ such that
\begin{equation*}
\big<Sx^{\ast },x-x^{\ast }\big>\geq 0,\ \forall \ x\in H.
\end{equation*}%
It is well known \cite{Byrne} that if $\gamma >0,$ then $x^{\ast }\in C$ is
a solution of $VIP(S,C)$ if and only if $x^{\ast }$ is a solution of fixed
point problem of $P_{C}\circ (I-\gamma S),$ where $P_{C}$ is the nearest
point projection onto $C.$\newline
Amongst many others results in \cite{Byrne}, it was proved that if $I-\gamma
S$ and $P_{C}\circ (I-\gamma S)$ are averaged nonexpansive operators, then,
under some additional assumptions, the iterative algorithm $%
\{x_{n}\}_{n=0}^{\infty }$ defined by
\begin{equation*}
x_{n+1}=P_{C}(I-\gamma S)x_{n},\ n\geq 0,
\end{equation*}%
converges weakly to a solution of the $VIP(S,C),$ if such solutions exist.%
\newline
In the case of averaged nonexpansive mappings, the problem of replacing the
weak convergence in the above result with strong convergence has received a
much attention of researchers. \newline
Our alternative is to consider $VIP(S,C)$ for enriched interpolative Kannan
type contraction operators, which are in general discontinuous operators in
contrast of nonexpansive operators which are always continuous. \newline
In this case, we shall have $VIP(S,C)$ with a unique solution, as shown in
the next theorem. Moreover, the considered algorithm (\ref{Viit}) will
converge strongly to the solution of the $VIP(S,C).$

\begin{thm}
Assume that for $\gamma >0$, $P_{C}(I-\gamma S)$ is enriched interpolative
Kannan type contraction operator. Then there exists $\lambda \in (0,1]$ such
that the iterative algorithm $\{x_{n}\}_{n=0}^{\infty }$ defined by
\begin{equation}
x_{n+1}=(1-\lambda )x_{n}+\lambda P_{C}(I-\gamma G)x_{n},\ n\geq 0,
\label{Viit}
\end{equation}%
converges strongly to the unique solution $x^{\ast }$ of the $VIP(S,C)$ for
any $x_{0}\in C.$

\begin{proof}
Since $C$ is closed, we take $X:=C$ and $T:=P_{C}(I-\gamma S)$ and apply
Corollary \ref{Shak}.
\end{proof}
\end{thm}

\begin{exmp}
Let $X=\mathbb{R}^{2}$ and for any $x=(x_{1},x_{2})$ and $y=(y_{1},y_{2})$
in $X$, the inner product is defined by
\begin{equation*}
\big<x,y\big>=x_{1}y_{1}+x_{2}y_{2}.
\end{equation*}%
Then $X$ equipped with the above inner product is a Hilbert space. The above
inner product gives the norm given by%
\begin{equation*}
\left\Vert x\right\Vert =(<x,y>)^{1/2}.
\end{equation*}
\newline
\ Define $S:X\rightarrow X$ by
\begin{equation*}
S(x)=\frac{(1,0)+x}{\gamma }\ \ \forall \ x\in X,
\end{equation*}%
where $\gamma >0$ be fixed real number. \newline
For a mapping $P_{C}:X\rightarrow C$ defined by
\begin{equation*}
P_{C}(x)=\left\{
\begin{array}{ll}
\frac{x}{\left\Vert x\right\Vert } & ;x\notin C \\
x & ;x\in C,%
\end{array}%
\right.
\end{equation*}%
where $C=\{x\in X:\left\Vert x\right\Vert \leq 1\},$ it is easy to check
that $P_{C}(I-\gamma S)$ is a $(a,0,\alpha )$ enriched interpolative Kannan
type contraction. \newline
By Corollary \ref{Shak}, $P_{C}(I-\gamma S)$ has a unique solution, which in
turns a solution for $VIP(S,C).$
\end{exmp}

\section{Conclusions}

\begin{enumerate}
\item We introduced a large class of contractive operators, called enriched
interpolative Kannan type operators, that includes usual interpolative
Kannan type operators and enriched Kannan operators.

\item We presented examples to show that the class of enriched interpolative
Kannan type operators strictly includes the interpolative Kannan type
operator in the sense that there exist operators which are not interpolative
Kannan type and belong to the class of enriched interpolative Kannan type
operators.

\item We studied the set of fixed point (Theorem \ref{main}) and constructed
an algorithm of Krasnoselskii type in order to approximate fixed point of
enriched interpolative Kannan type operators and we proved a strong
convergence theorem.

\item We also obtained Theorems \ref{math}, \ref{math2} and \ref{math3} for
a well-posedness, perodic point and Ulam-Hyers stability problem of the
fixed point problem for enriched interpolative Kannan type operators,
respectively.

\item As application of our main results (Corollary \ref{Shak}), we
presented two Krasnoselskii projection type algorithms for solving
variational inequality problems for the class of enriched interpolative
Kannan type operators, thus improving the existence and weak convergence
results for variational inequality problems in \cite{Byrne} to existence and
uniqueness as well as to strong convergence theorems.\newline
\end{enumerate}

\end{document}